\documentclass[11pt, reqno]{amsart}
\usepackage{a4wide,amsfonts,amsmath,amssymb,amsthm,epsfig,cite,graphicx,hyperref,
color, esint,fancyhdr, enumerate, latexsym, amsrefs}

\numberwithin{equation}{section}

\usepackage{tikz-cd}

\newtheorem{thm}{Theorem}[section]
\newtheorem{prop}[thm]{Proposition}
\newtheorem{lem}[thm]{Lemma}
\newtheorem{cor}[thm]{Corollary}
\theoremstyle{definition}

\newtheorem{rem}[thm]{Remark}

\newcommand{\mbb}{\mathbb}

\newcommand{\al}{\alpha}

\hypersetup{
    colorlinks,
    citecolor=blue,
    filecolor=black,
    linkcolor=red,
    urlcolor=black
}
\newcommand{\phih}{\phi_{\mathcal{H}}}
\newcommand{\xih}{\chi_{\mathcal{H}}}
\newcommand{\cdzor}{\bar{B}(Z_0, R)}

\newcommand{\xidj}{\chi_{D_j}}
\newcommand{\Cmincd}{\mathbb{C} \setminus \cdzor}

\newcommand{\Cdj}{c_{D_j}}

\newcommand{\sdj}{S_j}

\newcommand{\lrarw}{\longrightarrow}
\newcommand{\h}{\mathcal{H}}
\newcommand{\D}{\mathbb{D}}
\newcommand{\C}{\mathbb{C}}

\newcommand{\pin}{\pi_n}
\newcommand{\pit}{\tilde \pi}
\newcommand{\pint}{\tilde \pi_n}

\newcommand{\dnt}{\tilde D_n}

\newcommand{\du}{U \cap D}

\begin{document}
\title{Boundary behaviour of some conformal invariants on planar domains}

\author{Amar Deep Sarkar and Kaushal Verma}

\address{ADS: Department of Mathematics, Indian Institute of Science, Bangalore 560012, India}
\email{amarsarkar@iisc.ac.in}

\address{KV: Department of Mathematics, Indian Institute of Science, Bangalore 560 012, India}
\email{kverma@iisc.ac.in}



\begin{abstract}
The purpose of this note is to use the scaling principle to study the boundary behaviour of some conformal invariants on planar domains. The focus is on the Aumann--Carath\'{e}odory rigidity constant, the higher order curvatures of the Carath\'{e}odory metric and two conformal metrics that have been recently defined.
\end{abstract}

\maketitle

\section{Introduction}

The scaling principle in several complex variables provides a unified paradigm to address a broad array of questions ranging from the boundary behaviour of biholomorphic invariants to the classification of domains with non-compact automorphism group. In brief, the idea is to blow up a small neighbourhood of a smooth boundary point, say $p$ of a given domain $D \subset \mathbb C^n$ by a family of non-isotropic dilations to obtain a limit domain which is usually easier to deal with. The choice of the dilations is dictated by the Levi geometry of the boundary near $p$ and the interesting point here is that the limit domain is not necessarily unique. For planar domains, this method is particularly simple and the limit domain always turns out to be a half space if $p$ is a smooth boundary point. 

\medskip

The purpose of this note is to use the scaling principle to understand the boundary behaviour of some conformal invariants associated to a planar domain. We will focus on the Aumann--Carath\'{e}odory rigidity constant \cite{MainAumannCaratheodory}, the higher order curvatures of the Carath\'{e}odory metric \cite{BurbeaPaper}, a conformal metric arising from holomorphic quadratic differentials \cite{SugawaMetric} and finally, the Hurwitz metric \cite{TheHurwitzMetric}. These analytic objects have nothing to do with one another, except of course that they are all conformal invariants, and it is precisely this disparity that makes them particularly useful to emphasize the broad utility of the scaling principle as a technique even on planar domains.

\medskip

While each of these invariants requires a different set of conditions on $D$ to be defined in general, we will assume that $D \subset \mathbb C$ is bounded -- all the invariants are well defined in this case and so is $\lambda_D(z) \vert dz \vert$, the hyperbolic metric on $D$. Assuming this will not entail any great loss of generality but will instead assist in conveying the spirit of what is intended with a certain degree of uniformity. Any additional hypotheses on $D$ that are required will be explicitly mentioned. Let $\psi$ be a $C^2$-smooth local defining function for $\partial D$ near $p \in \partial D$ and let $ \lambda_D(z) \vert dz \vert $ denotes the hyperbolic metric on $ D $. In what follows, $\mathbb D \subset \mathbb C$ will denote the unit disc. The question is to determine the asymptotic behaviour of these invariants near $p$. Each of the subsequent paragraphs contain a brief description of these invariants followed by the corresponding results and the proofs will follow in subsequent sections after a description of the scaling principle for planar domains.

\medskip

\subsection{\textit{Higher order curvatures of the Carath\'{e}odory metric}:} Suita \cite{SuitaI} showed that the density $c_D(z)$ of the Carath\'{e}odory metric is real analytic and that its curvature
\[
\kappa(z) = - c_D^{-2}(z) \Delta \log c_D(z) 
\]
is at most $-4$ for all $z \in D$. In higher dimensions, this metric is not smooth in general. 

\medskip

For $j, k \ge 0$, let $\partial^{j \overline k} c_D$ denote the partial derivatives $\partial^{j+k} c_D/\partial z_j \partial \overline z_k$. Write $((a_{jk}))_{j, k \ge 0}^n$ for the $(n+1) \times (n+1)$ matrix whose $(j, k)$-th entry is $a_{jk}$. For $n \ge 1$, Burbea \cite{BurbeaPaper} defined the $n$-th order curvature of the Carath\'{e}odory metric $c_D(z) \vert dz \vert$ by
\[
\kappa_n(z: D) = -4c_D(z)^{-(n+1)^2} J^D_n(z)
\]
where $J^D_n(z) = \det ((\partial^{j \overline k} c_D))_{j, k \ge 0}^n$. Note that
\[
\kappa(z) = \kappa_1(z : D)
\]
which can be seen by expanding $J_1(z)$. Furthermore, if $f : D \rightarrow D'$ is a conformal equivalence between planar domains $D, D'$, then the equality
\[
c_D(z) = c_{D'}(f(z)) \vert f'(z) \vert
\]
upon repeated differentiation shows that the mixed partials of $c_D(z)$ can be expressed as a combination of the mixed partials of $c_{D'}(f(z))$ where the coefficients are rational functions of the derivatives of $f$ -- the denominators of these rational functions only involve $f'(z)$ which is non-vanishing in $D$. By using elementary row and column operations, it follows that
\[
J^D_n(z) = J^{D'}_n(f(z)) \vert f'(z) \vert^{n+1}
\]
and this implies that $\kappa_n(z: D)$ is a conformal invariant for every 
$n \ge 1$. If $D$ is equivalent to $\mathbb D$, a calculation shows that
\[
\kappa_n(z: D) = -4 \left( \prod_{k=1}^n k ! \right)^2
\]
for each $z \in D$. For a smoothly bounded (and hence of finite connectivity) $D$, Burbea \cite{BurbeaPaper} showed, among other things, that
\[
\kappa_n(z: D) \le -4 \left( \prod_{k=1}^n k ! \right)^2
\]
for each $z \in D$. This can be strengthened as follows:

\begin{thm}\label{T:HigherCurvature}
	Let $D \subset \mathbb C$ be a smoothly bounded domain. For every $p \in \partial D$
	\[
	\kappa_n(z: D) \rightarrow -4 \left( \prod_{k=1}^n k ! \right)^2
	\]
	as $z \rightarrow p$.
\end{thm}

\subsection{\textit{The Aumann--Carath\'{e}odory rigidity constant}:} Recall that the Carath\'{e}odory metric $c_D(z) \vert dz \vert$ is defined by
\[
c_D(z) = \sup \left\{ \vert f'(z) \vert : f : D \rightarrow \mathbb D \; \text{holomorphic and} \; f(z) = 0 \right\}.
\]
Let $D$ be non-simply connected and fix $a \in D$. Aumann--Carath\'{e}odory \cite{MainAumannCaratheodory} showed that there is a constant $\Omega(D, a)$, $0\le \Omega(D, a) < 1$, such that if $f$ is any holomorphic self-mapping of $D$ fixing $a$ and $f$ is 
{\it not} an automorphism of $D$, then $\vert f'(a) \vert \le \Omega(D, a)$. For an annulus $\mathcal A$, this constant was explicitly computed by Minda \cite{AumannCaratheodoryRigidityConstant} and a key ingredient was to realize that
\[
\Omega(\mathcal A, a) = c_{\mathcal A}(a)/\lambda_{\mathcal A}(a).
\]
The explicit formula for $\Omega(\mathcal A, a)$ also showed that 
$\Omega(\mathcal A, a) \rightarrow 1$ as $a \rightarrow \partial \mathcal A$. For non-simply connected domains $D$ with higher connectivity, \cite{AumannCaratheodoryRigidityConstant} also shows that
\[
c_{D}(a)/\lambda_{D}(a) \le \Omega(D, a) < 1.
\]
Continuing this line of thought further, Minda in \cite{HyperbolicMetricCoveringAumannCaratheodoryConstant} considers a pair of bounded domains $D, D'$ with base points $a \in D, b \in D'$ and the associated ratio
\[
\Omega(a, b) = \sup \{ \left(f^{\ast}(\lambda_{D'})/\lambda_D\right)(a): f \in \mathcal N(D, D'), f(a) = b \}
\]
where $\mathcal N(D, D')$ is the class of holomorphic maps $f : D \rightarrow D'$ that are {\it not} coverings. Note that
\[
\left(f^{\ast}(\lambda_{D'})/\lambda_D\right)(a) = \left( \lambda_{D'}(b) \;\vert f'(a) \vert \right) / \lambda_D(a).
\]
Among other things, Theorems 6  and 7 of \cite{HyperbolicMetricCoveringAumannCaratheodoryConstant} respectively show that 
\[
c_D(a)/\lambda_D(a) \le \Omega(a, b) < 1
\]
and 
\[
\limsup_{a \rightarrow \partial D} \Omega(a, b) = 1.
\]
Note that the first result shows that the lower bound for $\Omega(a, b)$ is independent of $b$ while the second one, which requires $\partial D$ to satisfy an additional geometric condition, is a statement about the boundary behaviour of 
$\Omega(a, b)$.  Here is a result that supplements these statements and emphasizes their local nature:

\begin{thm}\label{T:AumannCaratheodoryConstant}
	Let $D, D' \subset \mathbb C$ be a bounded domains and $p \in \partial D$ a $C^2$-smooth boundary point. Then $\Omega(D, z) \rightarrow 1$ as $z \rightarrow p$. Furthermore, for every fixed $w \in D'$,  $\Omega(z, w) \rightarrow 1$ as $z \rightarrow p$.
\end{thm} 

\subsection{\textit{Holomorphic quadratic differentials and a conformal metric}:} We begin by recalling a construction due to Sugawa \cite{SugawaMetric}. Let $R$ be a Riemann surface and $\phi$ a holomorphic $(m, n)$ form on it. In local coordinates $(U_{\alpha}, z_{\alpha})$, $\phi = \phi_{\alpha}(z_{\al}) dz^2_{\al}$ where  $\phi_{\alpha} : U_{\alpha} \rightarrow \mathbb C$ is a family of holomorphic functions satisfying
\[
\phi_{\alpha}(z_{\alpha}) = \phi_{\beta} (z_{\beta}) \left(\frac{d z_{\beta}}{d z_{\alpha}}\right)^m \left(\frac{d \overline z_{\beta}}{d \overline z_{\alpha}}\right)^n
\]
on the intersection $U_{\alpha} \cap U_{\beta}$. For holomorphic $(2,0)$ forms, this reduces to
\[
\phi_{\alpha}(z_{\alpha}) = \phi_{\beta} (z_{\beta}) \left(\frac{d z_{\beta}}{d z_{\alpha}}\right)^2
\]
and this in turn implies that
\[
\Vert \phi \Vert_1 = \int_R \vert \phi \vert 
\]
is well defined. Consider the space
\[
A(R) = \left\lbrace \phi = \phi(z) \;dz^2 \; \text{a holomorphic}\; (2,0)\; \text{form on}\; R \; \text{with} \; \Vert \phi \Vert_1 < \infty \right \rbrace 
\]
of integrable holomorphic $(2,0)$ forms on $R$. Fix $z \in R$ and for each local coordinate $(U_{\al}, \phi_{\al})$ containing it, let
\[
q_{R, \al}(z) = \sup \left\lbrace \vert \phi_{\al}(z) \vert^{1/2} : \phi \in A(R) \; \text{with} \; \Vert \phi \Vert_1 \le \pi \right\rbrace.
\]
Theorem 2.1 of \cite{SugawaMetric} shows that if $R$ is non-exceptional, then for each $z_0 \in U_{\al}$ there is a unique extremal differential $\phi \in A(R)$ ($\phi = \phi_{\al}(z_{\al}) dz^2_{\al}$ in $U_{\al}$) with $\Vert \phi \Vert_1 = \pi$ such that
\[
q_{R, \al}(z_0) = \vert \phi_{\al}(z_0) \vert^{1/2}.
\]
If $(U_{\beta}, \phi_{\beta})$ is another coordinate system around $z$, then the corresponding extremal differential $\phi_{\beta}$ is related to $\phi_{\alpha}$ as
\[
\overline{w'(z_0)} \;\phi_{\beta} = w'(z_0) \; \phi_{\alpha}
\]
where $w = \phi_{\beta} \circ \phi_{\al}^{-1}$. Hence $\vert \phi_{\al} \vert$ is intrinsically defined and this leads to the conformal metric $q_R(z) \vert dz \vert$ with $q_R(z) = q_{\al}(z)$ for some (and hence every) chart $U_{\al}$ containing $z$.

\medskip

It is also shown in \cite{SugawaMetric} that the density $q_R(z)$ is continuous, $\log q_R$ is subharmonic (or identically $-\infty$ on $R$) and $q_{\mathbb D}(z) = 1/(1 - \vert z \vert^2)$ -- therefore, this reduces to the hyperbolic metric on the unit disc. In addition, \cite{SugawaMetric} provides an estimate for this metric on an annulus. We will focus on the case of bounded domains.

\begin{thm}\label{T:SugawaMetric}
Let $D \subset \mbb C$ be a bounded domain and $p \in \partial D$ a $C^2$-smooth boundary point. Then
\[
q_D(z) \approx 1/{\rm dist} (z, \partial D)
\]
for $z$ close to $p$.
\end{thm}
Here and in what follows, we use the standard convention that $A \approx B$ means that there is a constant $C > 1$ such that $A/B, B/A$ are both bounded 
above by $C$. In particular, this statement shows that the metric 
$q_D(z) \vert dz \vert$ is comparable to the quasi-hyperbolic metric near $C^2$-smooth points. Thus, if $D$ is globally $C^2$-smooth, then $q_D(z) \vert dz \vert$ is comparable to the quasi-hyperbolic metric everywhere on $D$.

\subsection{\textit{The Hurwitz metric}:}
The other conformal metric that we will discuss here is the Hurwitz metric that has been recently defined by Minda \cite{TheHurwitzMetric}. We begin by recalling its construction which is reminiscent of that for the Kobayashi metric but differs from it in the choice of holomorphic maps which are considered: for a domain $D \subset \mathbb C$ and $a \in D$, let $\mathcal{O}(a, D)$ be the collection of all holomorphic maps $f : \mathbb{D} \rightarrow D$ such that $f(0) = a$ and $f'(0) > 0$. Let $\mathcal{O}^{\ast}(a, D) \subset \mathcal{O}(a, D)$ be the subset of all those $f \in \mathcal{O}(a, D)$ such that $f(z) \not= a$ for all $z$ in the punctured disc $\mathbb{D}^{\ast}$. Set
\[
r_D(a) = \sup \left\lbrace f'(0) : f \in \mathcal{O}^{\ast}(a, D) \right\rbrace.
\]
The Hurwitz metric on $D$ is $\eta_D(z) \vert dz \vert$ where
\[
\eta_D(a) = 1/r_D(a).
\]
Of the several basic properties of this conformal metric that were explored in \cite{TheHurwitzMetric}, we recall the following two: first, for a given $a \in D$, let $\gamma \subset D^{\ast} = D\setminus \{a\}$ be a small positively oriented loop that goes around $a$ once. This loop generates an infinite cyclic subgroup of
$\pi_1(D^{\ast})$ to which there is an associated holomorphic covering $G : \mathbb{D}^{\ast} \rightarrow D^{\ast}$. This map $G$ extends holomorphically to $G : \mathbb{D} \rightarrow D$ with $G(0) = a$ and $G'(0) \not= 0$. This covering depends only on the free homotopy class of $\gamma$ and is unique up to precomposition of a rotation around the origin. Hence, it is possible to arrange $G'(0) > 0$. Minda calls this the Hurwitz covering associated with $a \in D$. Using this it follows that every $f \in  \mathcal{O}^{\ast}(a, D)$ lifts to $\tilde f : \mathbb{D}^{\ast} \rightarrow \mathbb{D}^{\ast}$. This map extends to a self map of $\mathbb{D}$ and the Schwarz lemma shows that $\vert f'(0) \vert \le G'(0)$. The conclusion is that the extremals for this metric can be described in terms of the Hurwitz coverings. 

\medskip


\begin{thm}\label{T:TheHurwitzMetric}
Let $D \subset \mathbb{C}$ be bounded. Then $\eta_D(z)$ is continuous. Furthermore, if $p \in \partial D$ is a $C^2$-smooth boundary point, then
\[
\eta_D(z) \approx 1/{\rm dist} (z, \partial D)
\]
for $z$ close to $p$.
\end{thm}

A consequence of Theorem 1.3 and 1.4 is that both $q_D(z) \vert dz \vert$ and $\eta_D(z) \vert dz \vert$ are equivalent metrics near smooth boundary points.
\medskip

Finally, in section~7, we provide some estimates for the generalized curvatures of $ q_D(z) \vert dz \vert $ and $ \eta_D(z)\vert dz \vert $.

\section{Scaling of planar domains}
\noindent
The scaling principle for planar domains has been described in detail in \cite{ScalingInHigherDimensionKrantzKimGreen}. A simplified version which suffices for the applications presented later can be described as follows:
\medskip 
\\
Let $ D $ be a domain in $ \C $ and $ p \in \partial D $ a $ C^2 $-smooth boundary point. This means that there is a neighborhood $ U $ of $ p $ and a $ C^2 $-smooth real function $ \psi $ such that  
\[
U \cap D = \{\psi < 0\}, \quad U \cap \partial D = \{\psi = 0\} 
\]  
and 
\[
d \psi \neq 0 \quad \text{on} \quad U \cap \partial D.
\]
Let $ p_j $ be a sequence of points in $ D $ converging to $ p $. Suppose $ \tau(z)\vert dz \vert $ is a conformal metric on $ D $ whose behaviour near $ p $ is to be studied. The affine maps 

\begin{equation}\label{Eq:ScalingMap}
T_j(z) = \frac{z - p_j}{-\psi(p_j)}
\end{equation}
satisfy $ T_j(p_j) = 0 $ for all $ j $ and since $ \psi(p_j) \to 0$, it follows that the $ T_j $'s expand a fixed neighborhood of $ p $. To make this precise write
\[
\psi(z) = \psi(p) + 2Re\left( \frac{\partial \psi}{\partial z}(p)(z - p)\right)  +  o(\vert z - p\vert)
\] 
in a neighborhood of $ p $. Let $ K $ be a compact set in $ \C $. Since $ \psi(p_j) \to 0 $, it follows that $ T_j(U) $ is an increasing family of open sets that exhaust $ \C $ and hence $ K \subset T_j(U) $ for all large $ j $. The functions, by taking their Taylor series expansion at $ z = p_j $, 
\[
\psi \circ T^{-1}_{j}(z) = \psi \left( p_j + z\left( -\psi(p_j)\right) \right) 
= \psi(p_j) + 2Re\left( \frac{\partial \psi}{\partial z}(p_j)z \right) (-\psi(p_j)) + \psi(p_j)^2 o(1) 
\] 
are therefore well defined on $ K $ and the domains $ D_j' = T_j(U \cap D) $ are defined by 
\[
\psi_j(z) = \frac{1}{-\psi(p_j)}\psi \circ T_j^{-1}(z) = -1 + 2Re\left( \frac{\partial \psi}{\partial z}(p_j)z \right) + (-\psi(p_j))o(1).  
\]
It can be seen that $ \psi_j(z) $ converges to 
\[
\psi_{\infty}(z) = -1 + 2Re\left( \frac{\partial \psi}{\partial z}(p)z \right) 
\]
uniformly on $ K $ as $ j \to \infty $. At this stage, let us recall the Hausdorff metric on subsets of a metric space.
\medskip

Given a set $ S \subseteq \C^n $, let $ S_{\epsilon} $ denote the $ \epsilon $-neighborhood of $ S $ with respect to the standard Euclidean distance on $ \C^n $. The Hausdorff distance between compact sets $ X, Y \subset \C^n $ is given by 
\[
d_H(X,Y) = \inf\{\epsilon > 0 : X \subset Y_{\epsilon}\,\, \text{and} \,\, Y \subset X_{\epsilon}\}. 
\]
It is known that this is a complete metric space on the space of compact subsets of $ \C^n $. To deal with non-compact but closed sets there is a topology arising from a family of local Hausdorff semi-norms. It is defined as follows: fix an $ R > 0 $ and for a closed set $ A \subset \C^n $, let $ A^R = A \cap \overline B(0, R)$ where $ B(0, R) $ is the ball centred at the origin with radius $ R $. Then, for $ A, B \subset \C^n$, set
\[
d_H^{(R)}(A, B) = d_H\left(A^R, B^R\right).
\]
We will say that a sequence of closed sets $ A_n $ converges to a closed set $ A $ if there exists $ R_0 $ such that 
\[
\lim_{n \to \infty}d_H\left(A_n^R, A^R\right) = 0
\]
for all $ R \geq R_0 $. Since $ \psi_j \to \psi_{\infty} $ uniformly on every compact subset in $ \C $, it follows that the closed sets $ \overline{D_j'} = T_j(\overline{U \cap D})$  converge to  the closure of the half-space 
\[
\h = \{z : -1 +  2Re\left( \frac{\partial \psi}{\partial z}(p)z \right) < 0 \} = \{z: Re(\overline{\omega}z-1)<0\}
\]
where $\omega = (\partial\psi/\partial x)(p) + i(\partial\psi/\partial y)(p)$, in the Hausdorff sense as described above. As a consequence, every compact $ K \subset \h $ is eventually contained in $ D_j' $. Similarly, every compact $ K \subset \C \setminus \overline{\h} $ eventually has no intersection with $ \overline{D_j'} $. 
\medskip

It can be seen that the same property holds for the domains $ D_j = T_j(D) $, i.e, they converge to the half-space $ \h $ in the Hausdorff sense.
\medskip

Now coming back to the metric $ \tau(z)\vert dz \vert $, the pull-backs
\[
\tau_j(z) \vert dz \vert = \tau_D(T_j^{-1}(z)) \vert (T_j^{-1})^{\prime}(z)\vert \vert dz \vert
\]
are well-defined conformal metrics on $ D_j $ which satisfy 
\[
\tau_j(0) \vert dz \vert = \tau_D(p_j)(-\psi(p_j))\vert dz \vert.
\]
Therefore, to study $ \tau(p_j) $, it is enough to study $ \tau_j(0) $. This is exactly what will be done in the sequel.

\section{Proof of Theorem~\ref{T:HigherCurvature}}


It is known (see \cite[section~19.3]{InvariantMetricJarnicki} for example) that if $ D \subset \C $ is bounded and $ p \in \partial D $ is a $ C^2 $-smooth boundary point, then 
\[
\lim_{z \to p}\frac{c_{U \cap D}(z)}{c_D(z)} = 1
\]
where $ U $ is a neighborhood of $ p $ such that $ U \cap D $ is simply connected. Here is a version of this statement that we will need:
\begin{prop}
	Let $ D \subset \C $ be a bounded domain and $ p \in \partial D $ a $ C^2 $-smooth boundary point. Let $ U $ be a neighborhood of $ p $ such that $ U \cap D $ is simply connected. Let $ \psi $ be a defining function of $\partial D $ near the point $ p $. Then 
	\[
	\lim_{z \to p}c_{D}(z)(-\psi(z)) = c_{\h}(0).
	\]
\end{prop}
\begin{proof}
	Let $ \{p_j\} $ be a sequence in $ D $ which converges to $ p $. Consider the affine map  
	\[
	T_j(z) = \frac{z - p_j}{-\psi(p_j)}
	\]
	whose inverse is given by
	\[
	T_j^{-1}(z) = -\psi(p_j)z + p_j.
	\]
	Let $D_j = T_j(D)$ and $ D'_j = T_j(\du) $. 
	Note that $ \{D_j\} $ and $ D'_j $ converge to the half-space $ \h $ both in the Hausdorff and Carath\'eodory kernel sense. 
	\medskip
	
	Let $z \in \h$. Then $z \in D'_j$ for $j $ large.  Since $D'_j$ is a simply connected, there is a biholomorphic map $f_j : \mathbb{D} \lrarw D'_j$ with $f_j(0) = z$ and $f_j^{\prime}(0) > 0$. The domain $  D'_j $ converges to the half-space $ \h $ and therefore the Carath\'eodory kernel convergence theorem (see \cite{CaratheodoryKernelConvergence} for example), shows that $f_j$ admits a holomorphic limit $f : \mathbb{D} \lrarw \h$ which is a biholomorphism. Note that $f(0) = z$ and $f'(0) > 0$. We know that in the case of  simply connected domains, the Carath\'eodory and hyperbolic metric coincide and so 
	\[
	c_{D'_j}(z) = \lambda_{D'_j}(z) 
	\]
	for all large $ j $ and hence
	\[
	c_{\h}(z) = \lambda_{\h}(z).
	\]
	It is known that 
	\[
	\lambda_{D'_j}(z) = \frac{1}{f_j^{\prime}(0)} \,\, \mbox{and} \,\, \lambda_{\h}(z) = \frac{1}{f^{\prime}(0)}.
	\]
	From this we conclude that $c_{D'_j}(z)$ converges to $c_{\h}(z)$ as $j \to \infty$.
	\medskip
	
	Under the biholomorphism $ T_j^{-1} $, the pull back metric
	\[
	(T_j^{-1})^*(c_{\du})(z) = c_{D'_j}(z)
	\]
	for all $ z \in D_j$. That is
	\[
	c_{\du}(T_j^{-1}(z))\vert (T_j^{-1})^{\prime}(z)\vert  = c_{D'_j}(z).
	\]
	Putting $ z = 0 $, we obtain
	\[
	c_{\du}(p_j)(-\psi(p_j)) = c_{D'_j}(0).
	\]
	As we have seen above $ c_{D'_j}(z) $ converges to $ c_{\h}(z) $, for all $ z \in \h $, as $ j \to \infty $. Therefore, $ c_{\du}(p_j)(-\psi(p_j)) $, which is equal to $ c_{D_j}(0) $,  converges to $ c_{\h}(0) $ as $ j \to \infty $. Since $ \{p_j\} $ is an arbitrary sequence, we conclude that
	\[
	\lim_{z \to p}c_{\du}(z)(-\psi(z)) = c_{\h}(0).
	\]
	Since
	\[
	\lim_{z \to p}\frac{c_{U \cap D}(z)}{c_D(z)} = 1,
	\]
	we get
	\[
	\lim_{z \to p}c_{D}(z)(-\psi(z)) = c_{\h}(0).
	\]
\end{proof}  
\medskip

The Ahlfors map, the Szeg\"o kernel and the Garabedian kernel of the half space 
 
\[
\h = \{z : Re(\bar \omega z - 1) < 0\} 
\]
at $ a \in \h $
are given by
\[
f_{\h}(z, a) = \vert \omega \vert\frac{z - a}{2 - \omega \bar a - \bar \omega z},
\]
\[
S_{\h}(z, a) = \frac{1}{2 \pi}\frac{\vert \omega \vert}{2 - \omega \bar a - \bar \omega z}
\]
and
\[
L_{\h}(z, a) = \frac{1}{2 \pi}\frac{1}{z - a}
\]
respectively.
\medskip

Let $ D $ be a $ C^{\infty} $-smooth bounded domain. Choose $ p \in \partial D $ and a sequence $ p_j $ in $ D $ that converges to $ p $. The sequence of scaled domains $ D_j = T_j(D) $, where $ T_j $ are as in \ref{Eq:ScalingMap}, converges to the half-space $ \h $ as before.
\medskip

Fix $ a \in \h $ and note that $ a \in D_j $ for $ j $ large. Note that $ 0 \in D_j $ for $ j \geq 1 $. Let $ f_j(z, a) $ be the Ahlfors map such that $ f_j(a, a) = 0 $, $ f_j^{\prime}(a, a) > 0 $ and suppose that $ S_j(z, a) $ and $ L_j(z, a) $ are the Szeg\"o and Garabedian kernels for $ D_j $ respectively.
\begin{prop}\label{Prop:ConvergenceAhlforsSzegoGarabeidian}
	In this situation, the sequence of Ahlfors maps $ f_j(z, a) $ converges to $ f_{\h}(z, a) $ uniformly on compact subsets of $ \h $. The Szeg\H{o} kernels $ S_j(z, a) $ converge to $ S_{\h}(z, a) $ uniformly on compact subsets of $ \h $. Moreover, $ S_j(z, w) $ converges to the  $ S_{\h}(z, w) $ uniformly on every compact subset of $ \h \times \h $. Finally, the Garabedian kernels $ L_j(z, a) $ converges to $ L_{\h}(z, a) $ uniformly on compact subsets of $ \h \setminus \{a\} $.
\end{prop}
\begin{proof}
	In the proof of the previous proposition we have seen that $ c_{D_j}(a) $ converges $ c_{\h}(a) $ as $ j \to \infty $. By definition $f_j^{\prime}(a, a) = c_{D_j}(a)$ and $ f_{\h}^{\prime}(a, a) = c_{\h}(a) $. Therefore,  $f_j^{\prime}(a, a)$ converges to $f_{\h}^{\prime}(a, a)$ as $ j \to \infty $.
	\medskip
	
	Now, we shall show that $ f_j(z, a) $ converges to $f_{\h}(z, a)$ uniformly on compact subsets of $\h$. Since the sequence of the Ahlfors maps $\{f_j(z, a)\}$ forms a normal family of holomorphic functions and $ f_j(a, a) = 0 $, there exists a subsequence $\{f_{k_j}(z, a)\}$ of $\{f_j(z, a)\}$ that converges to a holomorphic function $f$ uniformly on every compact subset of the half-space $\h$. Then $f(a) = 0$ and, as $f_j^{\prime}(a)$ converges to $f_{\h}^{\prime}(a, a)$, so we have $f^{\prime}(a) = f_{\h}^{\prime}(a, a)$. Thus, we have $f: \h \longrightarrow \mathbb{D}$ such that $f(a) = 0$ and $f^{\prime}(a) = f_{\h}^{\prime}(a, a)$. By the uniqueness of the Ahlfors map, we conclude that $f(z) = f_{\h}(z, a)$ for all $ z \in \h $. Thus, from above, we see that every limiting function of the sequence $\{f_j(z, a)\}$ is equal to $f_{\h}(z, a)$. Hence, we conclude that $\{f_j(z, a)\}$ converges to $f_{\h}(z, a)$ uniformly on every compact subset of $ \h $. \smallskip

	Next, we shall show that $ S_j(\zeta, z) $ converge to $ S_{\h}(\zeta, z) $ uniformly on compact subsets of $\h \times \h$. First, we show that $\sdj(\zeta, z)$ is locally uniformly bounded. Let $z_0, \,\zeta_0 \in \h$ and choose $ r_0 > 0 $ such that the closed balls $ \overline{B}(z_0, r_0) $, $ \overline{B}(\zeta_0, r_0) \subset \h $. Since $ D_j $ converges to $ \h $, $ \overline{B}(z_0, r_0) $, $ \overline{B}(\zeta_0, r_0) \subset D_j $ for $ j $ large. By the monotonicity of the Carath\'eodory metric
	\[
	\Cdj(z) \leq \frac{r_0}{r_0^2 - |z|^2}
	\]
	for all $ z \in B(z_0,r_0)$, and 
	\[
	\Cdj(\zeta) \leq \frac{r_0}{r_0^2 - |\zeta|^2}
	\] 
	for all  $ \zeta \in  B(\zeta_0,r_0) $ and for $j$ large. Therefore, if $ 0 < r < r_0 $ and $ (z, \zeta)  \in \overline{B}(z_0, r_0) \times \overline{B}(\zeta_0, r_0) $, it follows that
	\[
	\Cdj(z) \leq \frac{r_0}{r_0^2 - r^2}
	\]
	and 
	\[
	\Cdj(\zeta) \leq \frac{r_0}{r_0^2 - r^2}
	\] 
	for $j$ large. Using the fact that $\Cdj(z) = 2\pi\sdj(z,z)$, we have
	\[
	\sdj(z, z) \leq \frac{1}{2 \pi} \frac{r_0}{r_0^2 - r^2}
	\]
	for all $ z \in \overline{B}(z_0,r) $ and 
	\[
	\sdj(\zeta, \zeta) \leq \frac{1}{2 \pi} \frac{r_0}{r_0^2 - r^2}
	\]
	for all $ \zeta \in \overline{B}(\zeta_0,r) $ and for $j$ large. By the Cauchy-Schwarz inequality,
	\[
	| \sdj(\zeta, z) |^2 \leq | \sdj(\zeta, \zeta)| | \sdj(z, z)|
	\]
	which implies
	\[
	| \sdj(\zeta, z) | \leq \frac{1}{2 \pi} \frac{r_0}{r_0^2 - r^2}
	\]
	for all  $ (\zeta, z) \in \overline{B}(\zeta_0,r) \times \overline{B}(z_0,r) $, and for all $ j \geq 1 $. 
	This shows that $\sdj(\zeta, z)$ is locally uniformly bounded. Hence the sequence $\{\sdj(\zeta,z)\}$, as a holomorphic function in the two variables $\zeta, \overline{z}$ is a normal family. Now, we claim that the sequence $\{\sdj(\zeta,z)\}$ converges to a unique limit. Let $S$ be a limiting function of $\{\sdj(\zeta,z)\}$. Since $S$ is holomorphic in $\zeta, \overline{z}$, the power series expansion of the difference $S - S_{\h}$ around the point $0$ has the form
	\[ 
	S(\zeta, z) - S_{\h}(\zeta, z) = \sum_{l,k =0}^{\infty}a_{l,k}\zeta^l\overline{z}^k. 
	\]   
	Recall that $2\pi S_{j}(z,z) = \Cdj(z)$ converges to $c_{\h}(z)$ as $ j \to \infty $ and $2\pi S_{\h}(z,z) = c_{\h}(z)$. From this we infer that $S(z, z) = S_{\h}(z, z)$ for all $ z \in \h $ and hence
	\[
	\sum_{l,k =0}^{\infty}a_{l,k}z^l\overline{z}^k = 0.
	\]
	By substituting $z = |z|e^{i\theta}$ in the above equation, we get
	\[
	\sum_{l,k =0}^{\infty}a_{l,k}|z|^{(l + k)}e^{i(l - k)\theta} = 0.
	\]
    and hence 
	\[
	\sum_{l +k = n}a_{l,k}e^{i(l - k)\theta} = 0 
	\]
	for all $ n \geq 1$. It follows that  $a_{l,k} = 0$ for all $l, k \geq 1 $. Hence we have 
	\[ 
	S(\zeta, z) = S_{\h}(\zeta, z)
	\] 
	for all $ \zeta, z \in \h $. So any limiting function of $\{\sdj(\zeta,z)\}$ is equal to $ S_{\h}(\zeta,z) $. This shows that $\{\sdj(\zeta,z)\}$ converges to $ S_{\h}(\zeta,z) $ uniformly on compact subsets of $ \h \times \h $.
	\medskip
	
	Finally, we show that the sequence of the Garabedian kernel functions $\{L_j(z, a)\}$ also converges to the Garabedian kernel function $ L_{\h}(z, a) $ of the half-space $ \h $ uniformly on every compact subset of $\h \setminus \{a\}$. This will be done by showing that $\{L_j(z, a)\}$ is a normal family and has a unique limiting function. To show that $\{L_j(z, a)\}$ is a normal family, it is enough to show that $L_j(z, a)$ is locally uniformly bounded on $ \h \setminus \{a\} $.
	\medskip
	 
	Since $ z = a $ is a zero of $ f_{\h}(z, a) $, it follows that for an arbitrary compact set $ K \subset \h \setminus \{a\} $, the infimum of $ \vert f_{\h}(z, a) \vert $ on $ K $ is positive. Let $ m > 0 $ be this infimum. As $ f_j(z, a) \to f_{\h}(z, a)$ uniformly on $ K $, 
	\[
	\vert f_j(z, a) \vert > \frac{m}{2}
	\] and hence 
	\[
	\frac{1}{\vert f_j(z, a) \vert} < \frac{2}{m} 
	\]
	for all $ z \in K $ and $ j $ large. Since $ S_j(z, a) $ converges to $ S_{\h}(z, a) $ uniformly on $ K $, there exists an $ M > 0 $ such that $ \vert S_j(z, a) \vert \leq M $ for $ j $ large.
	As
	\[
	L_j(z, a) = \frac{S_j(z, a)}{f_j(z, a)}
	\]
	for all $z \in D_j \setminus \{a\}$, we obtain  
	\[
	|L_j(\zeta, a)| \leq \frac{2M}{m}
	\]
	for all $\zeta \in  K$ and for $j$ large.
	This shows that $L_j(z, a)$ is locally uniformly bounded on $ \h \setminus \{a\} $, and hence a normal family. 
	\medskip
	
	Finally, any limit of $ L_j(z, a) $ must be $ S_{\h}\big/ f_{\h}(z,a) $ on $ \h \setminus \{a\} $ and hence $ L_j(z, a) $ converges to $ L_{\h}(z, a) $ uniformly on compact subsets of $ \h \setminus \{a\} $.  
\end{proof}

\medskip

\begin{rem}
	This generalizes the main result of \cite{SuitaI} to the case of a sequence of domains $ D_j $ that converges to $ \h $ as described above.
\end{rem}
\begin{prop}\label{Pr:UniformConvergenceOfCaraParDerCara}
	Let $\{D_j\}$ be the sequence of domains that converge to the half-space $\h$ in the Hausdorff sense as in the previous proposition. Then the sequence of Carath\'eodory metrics $c_{D_j}$ of the domains $D_j$ converges to the Carath\'eodory metric $c_{\h}$ of the half-space $ \h $ uniformly on every compact subset of $\h$. Moreover, all the partial derivatives of $c_{D_j}$ converge to the corresponding partial derivatives of $c_{\h}$, and the sequence of curvatures $ \kappa(z : D_j) $ converges to $  \kappa(z : \h) = -4 $, which is the curvature of the half-space $ \h $, uniformly on every compact subset of $ \h $.   	
\end{prop}
\begin{proof}
	Since 
	\[
	c_{D_j}(z) = 2 \pi S_j(z, z)
	\]
	and $ S_j(z, z) $ converges to $ S_{\h}(z, z) $ uniformly on compact subsets of $ \h $, we conclude that the sequence of Carath\'eodory metrics $c_{D_j}$ of the domains $D_j$ also converges to the Carath\'eodory metric $c_{\h}$ of the half-space $ \h $ uniformly on every compact subset of $\h$.
	\medskip
	
	To show that the derivatives of $ c_{D_j} $ converge to the corresponding derivatives of  $ c_{\h} $, it is enough to show the convergence in a neighborhood of a point in $\h$. 
	\medskip
	
	Let $ D^2 = D^2((z_0, \zeta_0); (r_1, r_2)) $ be a bidisk around the point $ (z_0, \zeta_0) \in \h \times \h$ which is relatively compact in $\h \times \h$. Then $ D^2 $ is relatively compact in $D_j \times D_j$ for all large $j$. Let $C_1 = \{z : |z -z_0| = r_1\}$ and  $C_2 = \{\zeta : |\zeta -\zeta_0| = r_2\}$.  
	\medskip
	
	Since $S_j(z, \zeta)$ is holomorphic in $z$ and antiholomorphic in $\zeta$, the Cauchy integral yields
	\[
	\frac{\partial^{m +n}S_j(z, \zeta)}{\partial z^m\partial \bar{z}^n } = \frac{m!n!}{(2\pi i)^2} \int_{C_1} \int_{C_2}\frac{S_j(\xi_1, \xi_2)}{(\xi - z)^{m + 1}(\overline{\xi - \zeta})^{n + 1}}d\xi_1 d\xi_2 
	\]
	and an application of the maximum modulus principle shows that
	\[
	\left|\frac{\partial^{m +n}S_j(z, \zeta)}{\partial z^m\partial \bar{z}^n }\right| \leq
	\frac{m!n!}{4\pi^2}\sup_{(\xi_1, \xi_2) \in D^2} |S_j(\xi_1, \xi_2)|\frac{1}{r_1^m r_2^n}.  
	\] 
	Now applying the above inequality for the function $S_j - S_{\h}$, we get
	\[
	\left|\frac{\partial^{m +n}}{\partial z^m\partial \bar{z}^n }( S_j - S_{\h} )(z, \zeta)\right| 
	\leq
	\frac{m!n!}{4\pi^2}\sup_{(\xi_1, \xi_2) \in D^2} |     S_j(\xi_1, \xi_2) - S_{\h}(\xi_1, \xi_2) |\frac{1}{r_1^m r_2^n}.
	\]
	Since $S_j(z, z) \to S_{\mathcal{H}}(z, z)$ uniformly on $ D^2 $, all the partial derivatives of $S_j(z, z)$ converge to the corresponding partial derivatives of $S_{\mathcal{H}}(z, z)$ uniformly on every compact subset of $D^2$.
	\medskip
	
	Recall that the curvature of the Carath\'eodory metric $ c_{D_j} $ is given by
	\[
	\kappa(z : c_{D_j}) = -c_{D_j}(z)^{-2}\Delta \log c_{D_j}(z)
 	\]
	which upon simplification gives
	\[
	-\Delta \log c_{D_j} = 4c_{D_j}^{-4}\left( \partial^{0\bar{1}}c_{D_j}\partial^{1\bar{0}}c_{D_j} - c_{D_j}\partial^{1\bar{1}}c_{D_j} \right)		
	\]
	where $ \partial^{i \bar j}c_{D_j} = \partial^{i + j}c_D\big/\partial^iz  \partial^j \bar z $ for $ i, j = 0, 1 $.
	Since all the partial derivatives of $ c_{D_j} $ converge uniformly to the corresponding partial derivatives of $ c_{\h} $,
	\[
	-\Delta \log c_{D_j} = 4c_{D_j}^{-4}\left( \partial^{0\bar{1}}c_{D_j}\partial^{1\bar{0}}c_{D_j} - c_{D_j}\partial^{1\bar{1}}c_{D_j} \right)		
	\]
	converges to 
	\[
	-\Delta \log c_{\h} = 4c_{\h}^{-4}\left( \partial^{0\bar{1}}c_{\h}\partial^{1\bar{0}}c_{\h} - c_{\h}\partial^{1\bar{1}}c_{\h} \right)		
	\]
	as $ j \to \infty $.
	\medskip
	
	Hence the sequence of curvatures 
	\[
	\kappa(z : c_{D_j}) = 4c_{D_j}^{-4}\left( \partial^{0\bar{1}}c_{D_j}\partial^{1\bar{0}}c_{D_j} - c_{D_j}\partial^{1\bar{1}}c_{D_j} \right)
	\]
	converges to $ \kappa(z,c_{\h}) = -4$ as $ j \to \infty $. 
\end{proof}
%

\medskip

\begin{proof}[Proof of Theorem~\ref*{T:HigherCurvature}]
	Note that the higher order curvatures of the Carath\'eodory metrics of the domains $ D_j $ are given by
	\[
	\kappa_n(z : c_{D_j}) =c_{D_j}(z)^{-(n + 1)^2}\big/ J^{D_j}_n(z).
	\]
	By appealing to the convergence $c_{D_j}$ and its partial derivatives to $c_{\h}$ and its corresponding partial derivatives uniformly on every compact  subset of $ \h $, we infer that  $\kappa_n(z, c_{D_j})$  converges to $\kappa_n(z : c_{\h}) = - 4 \left(\prod_{k = 1}^n k!\right)^2$ uniformly on every compact subset of $\h$. The fact that 
	\[
	\kappa_n(z : c_{\h}) = - 4 \left(\prod_{k = 1}^n k!\right)^2
	\]
	is a calculation using the Carath\'eodory metric on $ \h $.
\end{proof}


\section{Proof of Theorem~\ref{T:AumannCaratheodoryConstant}}

\begin{proof} 
	Scale $ D $ near $ p $ as explained earlier along a sequence $ p_j \rightarrow p $. Let $ D_j = T_j(D) $ as before and let $ \Omega_j $ be Aumann-Carath\'eodory constant of $ D_j $ and $ D' $. Then $ \Omega_j(0, w) = \Omega(p_j, w) $ for all $ j $. So, it suffices to study the behaviour of $ \Omega_j(0, w) $ as $ j \rightarrow \infty $.
	\medskip
	
	Let $ z \in D_j $ and $ \Omega_{j}(z, w) $ be Aumann-Carath\'eodory constant at $ z $ and $ w $. We shall show that $ \Omega_{j}(z, w) $ converges to $1$ as $ j \to \infty $.
	Let $ c_{D_j}$ and $ \lambda_{D_j}$ be Carath\'eodory and hyperbolic metric  on $D_j$ respectively, for all $j\geq 1$. Using the following inequality
	\[
	\Omega_{0} \leq \Omega \leq 1,
	\]
	we have 
	\begin{equation}\label{Eq:RatioCaraHyper}
	\frac{ c_{D_j}(z) }{ \lambda_{D_j}(z) }\leq \Omega_{j}(z, w) \leq 1.
	\end{equation}
	By Proposition~\ref{Pr:UniformConvergenceOfCaraParDerCara},
    $c_{D_j}(z)$ converges to $c_{\h}(z)$ uniformly on every compact subset of $\h$ as $ j \to \infty $. Again, using the scaling technique, we also have that the hyperbolic metric $\lambda_{D_j}(z) $ converges to $ \lambda_{\h}(z)$ on  uniformly every compact subsets of $\h$ as $ j \to \infty $. In case of simply connected domains, in particular for the half-space $ \h $, the hyperbolic metric and Carath\'eodory metric coincide, so we have $ c_{\h}(z) =  \lambda_{\h}(z) $.  Consequently, by (\ref{Eq:RatioCaraHyper}), we conclude that $ \Omega_{j}(z, w) $ converges to $1$, as $ j \to \infty $, uniformly on every compact subset of $ \h $. In particular, we have $ \Omega_{j}(0, w) \to 1 $ as  $ j \to \infty $. This completes the proof.  
\end{proof}
\begin{cor}
	Let $D \subset \C $ be a domain and  $ p \in \partial D $ is a $ C^2 $-smooth boundary point. Then the Aumann-Carath\'eodory rigidity constant
	\[
	\Omega_D(z) \rightarrow 1
	\]
	as $ z \rightarrow p $.
\end{cor}


\section{Proof of Theorem~\ref{T:SugawaMetric}}
\noindent
For a fixed $ p \in D $, let $ \phi_D(\zeta, p) $ be the extremal holomorphic differential for the metric $ q_D(z) \vert dz \vert $. Recall Lemma~2.2 from \cite{SugawaMetric} that relates how $ \phi_D(\zeta, p) $ is transformed by a biholomorphic map. Let $ F : D \longrightarrow D' $ be a biholomorphic map. Then 
\[
\phi_{D'}(F(\zeta), F(p)) \left(F^{\prime}(\zeta)\right)\frac{\overline{F^{\prime
}(p)}}{F^{\prime}(p)} = \phi_D(\zeta, p).
\]
\begin{lem}\label{L:ExtremalFunctionSugawaMetricHalfSpace}
	Let $\h = \{z \in \C : Re(\overline{\omega}z -1) < 0\}$ be a half-space, and let $\omega_0 \in \h$. Then the extremal function of $ q_{\h}(z) \vert dz \vert $ of $\h$ at $\omega_0 \in \h$ is
	\[
	\phi_{\h}(z, \omega_0) = |\omega|^2 \frac{(2 - \omega\overline{\omega_0} - \overline{\omega}\omega_0)^2}{(2 - \omega\overline{\omega_0} - \overline{\omega}z)^4}
	\]
	for all $z \in \h$.
\end{lem}
\begin{proof}
	The extremal function for $q_{\mathbb{D}}$ at the point $z = 0$ is 
	\begin{equation}\label{Eq: HurAtZeroDisk}
	\phi_{\mathbb{D}}(\zeta, 0) = 1.
	\end{equation}
	for all $\zeta \in \mathbb{D}$. Also
	\[
	f(z) = \frac{|\omega|(z - \omega_0)}{2 - \omega\overline{\omega_0} - \overline{\omega}z}
	\] 
	is the Riemann map between $ \h $ and the unit disk $ \D $. 
	By computing the derivative of the map $ f $ , we get
	\begin{equation}
	f^{\prime}(z) = |\omega| \frac{(2 - \omega\overline{\omega_0} - \overline{\omega}\omega_0)}{(2 - \omega\overline{\omega_0} - \overline{\omega}z)^2} 
	\end{equation}
	for all $ z \in \h $.
	By substituting the value of $f^{\prime}(z)$, $f^{\prime}(0)$ and $\overline{f^{\prime}(0)}$ in the transformation formula, we get
	\begin{equation}\label{Eq:TansForwithPhiD}
	\phi_{\h}(z, \omega) = |\omega|^2 \frac{(2 - \omega\overline{\omega_0} - \overline{\omega}\omega_0)^2}{(2 - \omega\overline{\omega_0} - \overline{\omega}z_0)^4}\phi_{\D}(f(z), f(\omega_0)).
	\end{equation}
	Note that $f(\omega_0) = 0$, therefore by (\ref{Eq: HurAtZeroDisk}), we have $\phi_{\D}(f(z), f(\omega_0)) = 1$. So from (\ref{Eq:TansForwithPhiD})  
	\begin{equation}
	\phi_{\h}(z, \omega_0) = |\omega|^2 \frac{(2 - \omega\overline{\omega_0} - \overline{\omega}\omega_0)^2}{(2 - \omega\overline{\omega_0} - \overline{\omega}z)^4}
	\end{equation}
	for all $z \in \h$.
\end{proof}

\begin{lem}\label{L:ConvergenceOfIntegral}
	Let the sequence of domains $ D_j $ converges to $ \h $ as before. Let $ z_j $ be a sequence in $ \h $ that converges to $ z_0 \in \h $. Let $ \phi_{\h,j} $ be the extremal function of $ q_{\h}(z)\vert dz \vert $ at $ z_j $ for all $ j $ and $ \phi_{\h} $ be the extremal function for $ q_{\h}(z)\vert dz \vert $ at $ z_0 $. Then  
	\[ 
	\int_{D_j}|\phi_j| \to \int_{\h}\vert \phi_{\h}\vert  = \pi 
	\]
	as $j \to \infty$.
\end{lem}
\begin{proof}
	By Lemma~\ref{L:ExtremalFunctionSugawaMetricHalfSpace}, we see that the extremal functions of $ q_{\h}(z) \vert dz \vert $ of the half-space $\h$ at the points $z_j$ and $z_0$ are given by
	\begin{equation}
	\phi_{\h, j}(z, z_j) = |\omega|^2 \frac{(2 - \omega\overline{z_j} - \overline{\omega}z_j)^2}{(2 - \omega\overline{z_j} - \overline{\omega}z)^4}
	\end{equation}
	and 
	\begin{equation}
	\phi_{\h}(z, z_0) = |\omega|^2 \frac{(2 - \omega\overline{z_0} - \overline{\omega}z_0)^2}{(2 - \omega\overline{z_0} - \overline{\omega}z)^4}
	\end{equation}
	respectively, for all $z \in \h$ and for all $j \geq 1$. By substituting $Z_j = \frac{2 - \omega\overline{z_j}}{\overline{\omega}}$ and $Z_0 = \frac{2 - \omega\overline{z_0}}{\overline{\omega}}$ in the above equations, we can rewrite the above equations as 
	\[
	\phi_{\h, j}(z,z_j) = \frac{|\omega|^2}{\overline{\omega}^4}\frac{(2 - \omega\overline{z_j} - \overline{\omega}z_j )^2}{(Z_j - z)^4}
	\]
	and 
	\[
	\phih(z,z_0) = \frac{|\omega|^2}{\overline{\omega}^4}\frac{(2 - \omega\overline{z_0} - \overline{\omega}z_0 )^2}{(Z_0 - z)^4}.
	\]
	respectively, for all $z \in \h$ and for all $j \geq 1$. 
	
%
%
%
    \medskip
    Note that $ Z_0 \notin \h $ and hence $ Z_j \notin \h $ for $ j $ large.	
	Define 
	\[
	\phi_j(z) =  \xidj(z)|\phi_{\h,j}(z, z_j)|  
	\]
	and 
	\[
	\phi(z) =  \xih(z) |\phih(z, z_0)|
	\]
	Here, $\chi_{A}$, for $A \subset \C$, its characteristic function.
	Note that $\phi_j$ and $ \phi $ are measurable functions on $ \C $ and $ \phi_j \to \phi $ point-wise almost everywhere on $\C$.
	\medskip
	
	Next, we shall show that there exists a measurable function $g$ on $\C$ satisfying
	\[
	\vert \phi_j \vert \leq g
	\]
	for all $j\geq 1$ and 
	\[
	\int_{\C}|g| < \infty.
	\]
	\medskip
	
	First, we note that 
	\[
	|Z_j - z| \geq \frac{R}{2}
	\]
	for all $j \geq 1$. Again, by the triangle inequality, we have  
	\[
	\left|\frac{|Z_0 - z|}{|Z_j - z|} - 1 \right| \leq \frac{|Z_0 - Z_j|}{|Z_j - z|} \leq 1
	\]
	for all $j \geq 1$ and for all $z \in \C \setminus B(Z_0, R)$. Therefore
	\[
	\frac{1}{|Z_j - z|} \leq \frac{2}{|Z_0 - z|}
	\]
	for all $z \in \C \setminus B(Z_0, R)$ and for all $j \geq 1$ and this implies that
	\[
	\frac{1}{|\omega^2|} \frac{|2 - \omega\overline{z_j} - \overline{\omega}z_j |^2}{|Z_j - z|^4} \leq \frac{16}{|\omega^2|}\frac{|2 - \omega\overline{z_j} - \overline{\omega}z_j |^2}{|Z_0 - z|^4}
	\] 
	for all $z \in \C \setminus B(Z_0, R)$ and for all $j \geq 1$.
	\medskip
	Note that there exists $ K > 0 $ such that  
	\[
	|2 - \omega\overline{z_j} - \overline{\omega}z_j |^2 \leq K
	\] 
	for all $j \geq 1$ since $ \{z_j \} $ converges. Hence
	\[
	|\phi_j(z)| \leq \frac{16 K}{|\omega^2|}\frac{1}{|Z_0 - z|^4}
	\]
	for all $z \in \C \setminus B(Z_0, R)$ and for all $j \geq 1$. If we set
	\[
	g(z) = \begin{cases} \frac{16 K}{|\omega^2|}\frac{1}{|Z_0 - z|^4}, &\mbox{if} \,\, z \in \C \setminus B(Z_0, R)\\
	0  & \mbox{if}\,\, z \in \cdzor
	\end{cases}
	\]
	we get that 
	\[
	|\phi_j(z)| \leq g(z)
	\]
	for all $z \in \C$ and for all $j \geq 1$. 
	\medskip
	
	Now note that 
	\begin{equation*}
	\begin{split}
	\int_{\C}g       &= \int_{\C \setminus B(Z_0, R)} \frac{16K}{|\omega|^2}\frac{1}{|Z_0 - z|^4}
	=  \frac{16K}{|\overline{\omega}^2|}\int_{\Cmincd} \left|\frac{1}{(Z_0 - \zeta)^4}\right| \\
	&=  \frac{16K}{|\overline{\omega}^2|} \int_{R}^{\infty} \int_0^{2\pi} \frac{1}{r^3}  
	=  \frac{16K}{|\overline{\omega}^2|} \frac{1}{R^2} < \infty.
	\end{split}
	\end{equation*}
	The dominated convergence theorem shows that 
	\[
	\int_{\C}\phi_j
	\to
	\int_{\C}\phi_{\h} 
	\]
	as $ j \to \infty $. 
	However, by construction
	\[
	\int_{\C}\phi_j  = \int_{D_j}\phi_{\h, j}(z,z_j)
	\]
	and
	\[
	\int_{\C}\phi  = \int_{\h}\phi_{\h}(z, z_0)
	\]
	and by definition. Hence
	\[
	\int_{\h}\phi_{\h}(z, z_0) = \pi
	\]
	 and this completes the proof.
\end{proof}
\medskip

\begin{prop}\label{Pr:ConvergenceOfTheSugawaMetricUnderScaling}
	Let $\{D_j\}$ be the sequence of domains that converges to the half-space $\h$ as in Proposition~\ref{Prop:ConvergenceAhlforsSzegoGarabeidian}. Then $q_{D_j}$ converges to $q_{\h}$ uniformly on compact subsets of $ \h $.  
\end{prop}
\begin{proof}
	If possible, assume that $q_{D_j}$ does not converge to $q_{\h}$ uniformly on every compact subset of $\h$. 
	Then there exists a compact subset $K$ of $\h$ -- without loss of generality, we may assume that $K \subset D_j$ for all $j \geq 1$ --  $\epsilon_{0} >0 $, a sequence of integers $\{ k_j \}$ and points $\{ z_{k_j}\}\subset K$ such that
	\[
	|q_{D_{k_j}}(z_{k_j})- q_{\mathcal{H}}(z_{k_j})|>\epsilon_0.
	\]      
	Since $K$ is compact, we assume that $z_{k_j}$ converges to a point $z_0 \in K$.
	Using the continuity of $ q_{\h}(z) \vert dz \vert $, we have 
	\[
	|q_{\mathcal{H}}(z_{k_j})- q_{\mathcal{H}}(z_0)|< \epsilon_0\big/2
	\] 
	for $j$ large, which implies 
	\[
	|q_{D_{k_j}}(z_{k_j})- q_{\mathcal{H}}(z_0)|> \epsilon_0\big/2
	\]
	by the triangle inequality.
	There exist extremal functions $\phi_{k_j} \in A(D_{k_j})$ with $ \Vert \phi_{k_j} \Vert \leq \pi $ such that 
	\[
	q_{D_{k_j}}^2 = \phi_{k_j}(z_{k_j}).
	\] 
	We claim that the collection $\{\phi_{k_j}\}$ is a normal family. To show this, it is enough to show $\{\phi_{k_j}\}$ is locally uniformly bounded.
	\medskip
	
	Let $\zeta_0 \in \h$ and $R>0$ such that $\overline{B}(\zeta_0, 2R) \subset \h$. Then $\overline{B}(\zeta_0, 2R) \subset D_j$ for all $j$ large. By the mean value theorem, we have      
	\begin{equation*}
	\begin{split}
	|\phi_{k_j}(\zeta)| = \frac{1}{2\pi} \left|\int_{0}^{2\pi}\phi_{k_j}(\zeta + re^{\theta})d\theta\right| 
	\leq \frac{1}{2\pi} \int_{0}^{2\pi}|\phi_{k_j}(\zeta + re^{\theta})|d\theta
	\end{split}
	\end{equation*}
	for $ 0 < r < R $.
	This implies
	\begin{equation*}
	\begin{split}
	\int_0^R|\phi_{k_j}(\zeta)|rdr &\leq \frac{1}{2\pi} \int_0^R\int_{0}^{2\pi}|\phi_{k_j}(\zeta + re^{\theta})|r d\theta dr 
	\leq \frac{1}{ 2\pi } \int_0^R\int_{0}^{2\pi}|\phi_{k_j}(\zeta + re^{\theta})|r d\theta dr,
	\end{split}
	\end{equation*}
	that is, 
	\begin{equation*}
	\begin{split}
	|\phi_{k_j}(\zeta)|
	&= \frac{1}{\pi R^2} \int\int_{B(\zeta, R)}|\phi_{k_j}(z)| 
	\leq \frac{1}{\pi R^2} \int\int_{D_j}|\phi_{k_j}(z)|
	= \frac{\pi}{\pi R^2} = \frac{1}{R^2}
	\end{split}
	\end{equation*} 
	for all $\zeta \in B(\zeta_0, R)$. This proves that the family $\{\phi_{k_j}\}$ is locally uniformly bounded on $B(\zeta_0, R)$ and hence a normal family. 
	\medskip
	
	Next, we shall show that
	\[
	\limsup_j q_{D_j}(z_j) \leq q_{\mathcal{H}}(z_0).
	\] 
	First, we note that the finiteness of the integral $\int_{D_j}|\phi_{k_j}| = \pi$ implies that no limiting function of the sequence $\{\phi_{k_j}\}$ diverges to infinity uniformly on any compact subset of $\h$.
	\medskip
	
	Now, by definition, we see that there exists a subsequence $\{\phi_{l_{k_j}}\}$ of the sequence $\{\phi_{k_j}\}$ such that $|\phi_{l_{k_j}}(z_0)|$ converges to $\limsup_{j}|\phi_{k_j}(z_0)|$ as $ j \to \infty$. Let $\phi$ be a limiting function of $\{\phi_{l_{k_j}}\}$. For simplicity, we may assume that $\{\phi_{l_{k_j}}\}$ converges to $\phi$ uniformly on every compact subset of $\h$. We claim that $\phi \in A(\h)$ and $ \Vert \phi \Vert \leq \pi $. That is, we need to show 
	\[
	\int_{\h}|\phi| \leq \pi.
	\]
	To show this we take an arbitrary compact subset $K$ of $\h$. Then $K \subset D_j$ and, by the triangle inequality,
	\[
	\int_{K}|\phi| \leq \int_{K} |\phi - \phi_{l_{k_j}}| + \int_{K}|\phi_{l_{k_j}}|
	\]
	for $j $ large. 
	This implies,
	\[
	\int_{K}|\phi| \leq \int_{K} |\phi - \phi_{l_{k_j}}| + \int_{D_j}|\phi_{l_{k_j}}|
	\]
	for $j$ large.
    Since $\int_{D_j}|\phi_{l_{k_j}}| = \pi$ and $\{\phi_{l_{k_j}}\}$ converges to $\phi$ uniformly on $K$, $\int_{K} |\phi - \phi_{l_{k_j}}|$ converges to $0$ as $ j \to \infty$. Hence
	\[
	\int_{\h}|\phi| \leq \pi.
	\]
	This shows that $\phi(z)  $ is a candidate in the family that defines $ q_{\h}(z) $.    
	\medskip
	
	By definition of $ q_{\h}(z) \vert dz \vert $, we have 
	\[
	\phi(z_0) \leq q_{\mathcal{H}}^2(z_0)
	\]  
	which implies 
	\[
	\limsup_j\phi_{k_j}(z_j) \leq q_{\mathcal{H}}^2(z_0).
	\]
	Now, by substituting $q_{D_j}^2(z_j)$ in place of $\phi_{k_j}(z_j)$ above, we obtain  
	\[
	\limsup_j q_{D_j}^2(z_j) \leq q_{\mathcal{H}}^2(z_0)
	\]
	or
	\begin{equation}\label{Eqn: Eqsup}
	\limsup_j q_{D_j}(z_j) \leq q_{\mathcal{H}}(z_0).
	\end{equation}
	\medskip
	
	Next, we show
	\[
	q_{\mathcal{H}}(z_0) \leq \liminf_j q_{D_j}(z_0),
	\]
	and this will lead to a contradiction to our assumption.
	\medskip
	
	As we have seen in  Lemma~\ref{L:ExtremalFunctionSugawaMetricHalfSpace}, the extremal functions of $ q_{\h}(z) \vert dz \vert $ at the points $z_j$ and $z_0$ of the half-space $\h$ are given by rational functions with poles at $Z_j = (2 - \omega\overline{z_j})\big/\overline{\omega}$ and $Z_0 = (2 - \omega\overline{z_0})\big/\overline{\omega}$ respectively. The Hausdorff convergence of $ D_j $ to $ \h $ as $ j \to \infty $ implies that there exists $R>0$ such that 
	$Z_j \in B(Z_0, R\big/2)$ and $\overline{B}(Z_0, R) \subset \C \setminus \overline{D_j}$  for $j$ large. This ensures
	\[
	\psi_j(z) = \pi \phi_j(z)\big / M_j
	\]
	is a well-defined holomorphic function on $ D_j $,
	where $ \phi_j $ is an extremal function of $ q_{\h}(z) \vert dz \vert $ at $ z_j $ and
	\[
	M_j = \int_{D_j} |\phi_j|.
	\]
	By Lemma~\ref{L:ConvergenceOfIntegral}, $ M_j < \infty $ for $j$ large. From the definition of the function $\psi_j$ it follows that $\int_{D_j} |\psi_{\mathcal{H}}| = \pi$ for all $j \geq 1$. Therefore, 
	\[
	\psi_j \in A_0(D_j) = \left\{ |\phi(z)|^2 : \phi \in A(D_j) \, \text{with} \, ||\phi||_1 \leq \pi \right \}
	\]
	and hence 
	\[
	\pi \psi_{\h}(z_j)\big/ M_j = \psi_j(z_j) \leq q_{D_j}^2(z_j).
	\]
    \medskip
                  
	Since $M_j$ converges to $\pi$ as $ j \to \infty$, Lemma~\ref{L:ConvergenceOfIntegral} and by the formulae of the extremal functions of $ \h $, it is seen that $\phi_j(z_j)$ converges to $\phi(z_0)$, the extremal function of the half-space at the point $ z_0 $, as $j \to \infty$. Therefore, taking liminf both sides,
	\[
	q_{\mathcal{H}}^2(z_0) \leq \liminf_j q_{D_j}^2(z_j).
	\]
	That is,
	\begin{equation}\label{Eqn: Eqinf}
	q_{\mathcal{H}}(z_0) \leq \liminf_j q_{D_j}(z_j).
	\end{equation}
	
	By (\ref{Eqn: Eqsup}) and (\ref{Eqn: Eqinf}), we conclude that 
	\[
	\lim_{j \to \infty} q_{D_j}(z_j) = q_{\mathcal{H}}(z_0)
	\]
	which contradicts the assumption that
	\[
	|q_{D_{k_j}}(z_{k_j})- q_{\mathcal{H}}(z_0)|> \epsilon_0\big /2.
	\]
\end{proof}
\medskip

\begin{cor}\label{C:AsymptoticqD}
	Let $ D \subset \C$ be a bounded domain. Suppose $ p \in \partial D $ is a $ C^2 $-smooth boundary point. Then 
	\[
	\lim_{z \to p}q_D(z)(-\psi(z)) = q_{\h}(0) = \vert \omega \vert \big/2.
	\]
\end{cor}
\begin{proof}
Let $ D_j = T_j(D) $ where $ T_j(z) = (z - p_j)\big/(-\psi(p_j)) $. By Proposition~\ref{Pr:ConvergenceOfTheSugawaMetricUnderScaling}  $ q_{D_j} \to q_{\h} $ on compact subsets of $ \h $. In particular, $q_{D_j}(0) \to q_{\h}(0) $ as $ j \to \infty $. But $ q_{D_j}(0) = q_{D}(p_j) (-\psi(p_j)) $ and hence $ q_{D_j}(0) = q_{D}(p_j) (-\psi(p_j)) \to q_{\h}(0) $. Since $ p_j $ is an arbitrary sequence converging to $ p $, hence it follows that
	\[
	\lim_{z \to p}q_D(z)(-\psi(z)) = q_{\h}(0) = \vert \omega \vert \big / 2.
	\]
	
\end{proof}

\begin{proof}[Proof of Theorem~\ref{T:SugawaMetric}]
As a consequence of the Corollary~\ref{C:AsymptoticqD}, we have
\[
q_D(z) \approx 1\big /{\rm dist}(z, \partial D).
\]
\end{proof}


\section{Proof of Theorem~\ref{T:TheHurwitzMetric}}
\noindent
The continuity of $ \eta_D $ is a consequence of the following observation.

\begin{thm}\label{T:contHurM}
	Let $D \subset \mathbb{C}$ be a domain. Fix $a\in D$ and let $\{a_n\}$ be a sequence in $ D$ converging to $a$. Let $G_n,\, G$ be the normalized Hurwitz coverings at $a_n$ and $a$ respectively.Then $G_n$ converges to $G$ uniformly on compact subsets of $\mathbb{D}$. 
\end{thm}
\medskip
\noindent
The proof of this requires the following lemmas.
\begin{lem}\label{L:PiPosCon}
	Let $ y_0 \in \D^* $ and $ \{y_n\} $ a sequence in $ \D^* $ converging to $ y_0 $. 
	Let $\pi_n : \mathbb{D} \longrightarrow \mathbb{D}^{*}$ and $\pi_0 : \mathbb{D} \longrightarrow \mathbb{D}^{*}$ be the unique normalized coverings satisfying $\pi_n(0) = y_n$, $\pi_0(0) = y_0$   and $\pi^{\prime}_n(0) > 0$, $\pi^{\prime}_0(0) > 0$. Then  $\{\pi_n \}$ converges to $\pi_0$ uniformly on compact subsets of $ \D $. 
\end{lem}

\begin{proof}
	Note that the hyperbolic density $\lambda_{\mathbb{D}^{*}}$ is continuous and hence $\lambda_{\mathbb{D}^{*}}(y_n) \to \lambda_{\mathbb{D}^{*}}(y)$. Also $\lambda_{\mathbb{D}^{*}} > 0 $ and satisfies
	\[
	 \lambda_{\mathbb{D}^{*}}(y_n) = 1 \big/ \pi^{\prime}_n(0),
	\]
	\[
	\lambda_{\mathbb{D}^{*}}(y_0) = 1\big/\pi^{\prime}_0(0).
	\]
	Let $ \pi_{\infty}$ be the limit point of the family $\{\pi_n\}$ which is a normal. Then $ \pi_{\infty} : \D \longrightarrow \D^*$ and $ \pi_{\infty } (0) =y_0 \in \D^* $. If the image $ \pi_{\infty}(\D^*) $ intersects $ \partial \D $ then $ \pi_{\infty} $ must be constant, i.e., $ \pi_{\infty}(z) = e^{i \theta_0} $ for some $ \theta_0 $. This can not happen since $ \pi_{\infty}(0) = y_0 \in \D^* $. If the image $ \pi_{\infty}(\D^*) $ contains the origin, Hurwitz's theorem shows that $ \pi_n $ must also contains the origin for $ n $ large. Again this is not possible since $ \pi_n(\D) = \D^* $. It follows that $ \pi_{\infty} : \D \longrightarrow \D^* $ with $ \pi_{\infty}(0) = y_0 $.  
    \medskip
    
	Let $ \tilde\pi_{\infty} : \D \longrightarrow \D $ be the lift of $ \pi_{\infty} $ such that $ \tilde\pi_{\infty}(0) =  0 $. By differentiating the identity,
	\[
	\pi_0 \circ \tilde\pi_{\infty} = \pi_{\infty},
	\]
		we obtain
		\[
		\pi_0^{\prime}(\tilde\pi_{\infty}(z)) \tilde\pi_{\infty}^{\prime}(z) = \pi_{\infty}^{\prime}(z)
		\]  
		and evaluating at $ z = 0 $ gives 
		\[
		\pi^{\prime}(0) \tilde \pi_{\infty}^{\prime}(0) = \pi_{\infty}^{\prime}(0).
		\]
		Now observe that $ \pi_n^{\prime}(0) \to \pi_{\infty}^{\prime}(0) $ and since $1 \big/ \pi_{n}^{\prime} (0) = \lambda_{\mathbb{D}^{*}}(y_n) \to \lambda_{\mathbb{D}^{*}}(y_0) $ it follows $ \pi_{\infty}^{\prime}(0) = 1 \big/ \lambda_{\mathbb{D}^{*}}(y_0) = \pi_0^{\prime}(0) > 0 $. Therefore, $ \tilde \pi_{\infty}^{\prime}(0) = 1  $ and hence $ \pi_{\infty}^{\prime}(z) = z $ by the Schwarz lemma. As a result, $ \pi_{\infty} = \pi_0 $ on $ \D $. If follows that $ \{ \pi_n \} $ has unique limit, namely $ \pi_0 $.         
\end{proof}

\begin{lem}\label{L:UniformConvergenceOfPunCoveringMpas}
	Let $\pi_n : \mathbb{D} \longrightarrow \mathbb{D}^{*}$  be a sequence of holomorphic coverings. Let $\pi$ be a non-constant limit point of the family  $\{\pi_n\}$. Then $\pi : \mathbb{D} \longrightarrow \mathbb{D}^{*}$ is a covering.
\end{lem}

\begin{proof}
	Suppose that $ \pi_{n_k} \to \pi $ uniformly on compact subsets of $ \D $. As in the previous lemma, $ \pi : \D \longrightarrow \D^* $ since $ \pi $ is assumed to be non-constant. In particular, $ \pi_{n_k}(0) \to \pi(0)  \in \D^*$. Let 
	\[
	\phi_n(z) = e^{i \theta_n}z
	\]
	where $ \theta_n = Arg(\pi_{k_n}^{\prime}(0)) $ and consider the compositions
	\[
	\tilde \pi_{k_n} = \pi_{k_n} \circ \phi_n
	\]
	which are holomorphic covering of $ \D^* $ satisfying $ \tilde \pi_{k_n}^{\prime}(0) > 0$ and $ \tilde \pi_{k_n}(0) = \pi_{k_n}(0) \to  \pi(0) \in \D^* $. By Lemma~\ref{L:PiPosCon}, $  \tilde \pi_{k_n} \to \tilde \pi  $ where $ \tilde \pi : \D \longrightarrow \D^* $ is a holomorphic covering with $ \tilde \pi(0) = \pi(0)  $ and $ \tilde \pi^{\prime}(0) > 0 $. By passing to a subsequence, $ \phi_n(z) \to \phi(z) $ whree $ \phi(z) = e^{i\theta_0}z  $ for some $ \theta_0 $. As a result, 
	\[
	\tilde \pi = \pi \circ \phi
	\]
	and this shows that $ \pi : \D \longrightarrow \D^* $ is a holomorphic covering.
\end{proof}
\medskip

\begin{lem}\label{L:HyperbolicCoveringConvergenc}
	Let $D \subset \C$ be bounded. Fix $ a \in D $ and let $\{a_n\}$ be a sequence in $ D $ converging to $a$. Set $ D_n = D\setminus \{x_n\}$ and $ D_0 = D\setminus \{a\} $. Fix a base point $p \in D \setminus \{a, a_1, a_2, \dots \}$. Let $\pi_n : \mathbb{D} \longrightarrow D_n $ and $\pi_0 : \mathbb{D} \longrightarrow D_0 $ be the unique normalized coverings such that $\pi_n(0) = \pi_0(0) = p $ and $\pi^{\prime}_n(0), \, \pi^{\prime}_0(0) > 0$. Then $\pi_n \to \pi_{0}$ uniformly on compact subsets of $\mathbb{D}$. 
\end{lem}
\begin{proof}
	Move $ p $ to $ \infty $ by $ T(z) = 1\big/(z - p) $ and let $\dnt = T(D_n)$, $\tilde D_0 = T(D_0)$. Then 
	\[
	\pint = T \circ \pin : \D \longrightarrow \dnt 
	\]
	and 
	\[\pit = T \circ \pi : \D \longrightarrow \tilde D_0 
	\]
	are coverings that satisfy
	\[
	\lim_{z \to 0} z\pint(z) = \lim_{z \to 0} \frac{z}{\pin(z) - \pin(0)} = \frac{1}{\pin^{\prime}(0)} > 0 
	\]
	and 
	\[
	\lim_{z \to 0} z\pit(z) = \lim_{z \to 0} \frac{z}{\pi(z) - \pi(0)} = \frac{1}{\pi^{\prime}(0)} > 0. 
	\]
	It is evident that the domains $ \dnt $ converge to $ \tilde D_0 $ in the Carath\'eodory kernel sense and Hejhal's result \cite{Hejhal} shows that
	\[
	\pit_n \to \tilde \pi_0
	\] uniformly on compact subsets of $ \D $. This completes the proof.   
	
\end{proof}
\begin{proof}[Proof of the Theorem~\ref{T:contHurM}]
	Fix a point $a \in D$ and let $ a_n \to a $. Let $ G_n, \, G$ be the normalized Hurwitz coverings at $ a_n, a $ respectively. Since $ D $ is bounded, $ \{ G_n \} $ is a normal family. Assume that $ G_n \to \tilde G $ uniformly on compact subsets of $ \D $. 
	\medskip
	
	By Theorem~6.4 of \cite{TheHurwitzMetric},

	\begin{equation}\label{Eq:bilipchitzCondition}
	1\big/8 \delta_D(z) \leq \eta_D(z) \leq 2\big/ \delta_D(z)
	\end{equation}
	where $ \delta_D(z) = {\rm dist}(z, \partial D) $. By definition, $ \eta_D(a_n) = 1\big/G_n^{\prime}(0) $ which gives
	\[
	2 \delta_D(a_n) \leq G_n^{\prime}(0) \leq 8 \delta_D(a_n)
	\]
	
	for all $ n $ and hence
	\[
	2 \delta_D(a) \leq \tilde G^{\prime}(0) \leq 8 \delta_D(a).
	\]
	Since $ \delta_D(a_n),\, \delta_D(a) $ have uniform positive lower and upper bounds, it follows that $ G_n^{\prime}(0)  $ and $ \tilde G^{\prime}(0) $ admits uniform lower and upper bounds as well. We will now use the following fact which is a consequence of the inverse function theorem. 
	\medskip
	
	{\bf \textit{Claim}}: Let $ f : \Omega \longrightarrow \C $ be holomorphic and suppose that $ f^{\prime}(z_0) \neq 0 $ for some $ z_0 \in 
	\Omega $. Then there exists $ \delta > 0 $ such that $ f : B(z_0, \delta) \longrightarrow G_n^{\prime}(0) $ is biholomorphic and $ B(f(z_0), \delta \vert f^{\prime}(z_0) \vert \big/2) \subset f(B(z_0, \delta))$. 
	\medskip
	
	To indicate a short proof of this claim, let $ \delta > 0 $ be such that 
	\begin{equation*}\label{Eq:DerIneq}
	\vert f^{\prime}(z) - f^{\prime}(z_0) \vert < \vert f^{\prime
	}(z_0) \vert\big/2
	\end{equation*}
	for all $ \vert z - z_0 \vert < \delta $. Then
	$ g(z) = f(z) - f^{\prime}(z_0)z $ satisfies 
	\[
	\vert g^{\prime}(z) \vert \leq 1\big/2 \vert f^{\prime}(z_0) \vert 
	\]
	and hence
	\begin{equation*}\label{Eq:LipCon}
	\vert f(z_2) - f(z_1) - f^{\prime}(z_0)z(z_2 - z_1) = 
	\vert g(z_2) - g(z_1)  \vert \leq  \vert f^{\prime}(z_0)\vert \vert z_2 - z_1 \vert \big/2.
    \end{equation*}
	This shows that $ f $ is injective on $ B(z_0, \delta) $. Finally, if $ w \in B(f(z_0), \delta \vert f^{\prime}(z_0)\vert \big/2) $, then 
	\[
	z_k = z_{k - 1} - \frac{w - f^{\prime}(z_{k - 1})}{f^{\prime}(z_0)}
	\]
	defines a Cauchy sequence which is compactly contained in $ B(z_0, \delta) $. It converges to $ \tilde z \in B(z_0, \delta)  $ such that $ f(\tilde z) = w $. This shows that $ B(f(z_0), \delta \vert f^{\prime}(z_0) \vert \big/2) \subset f(B(z_0, \delta))$.
	\medskip
	
	Let $ m, \delta > 0 $ be such that
	\[
	\vert \tilde G^{\prime}(z) - \tilde G^{\prime}(0) \vert < m \big/2 < m \leq \frac{\vert \tilde G^{\prime}(0) \vert}{2}
	\] 
	for $ \vert z \vert < \delta $. Since $ G_n $ converges to $ \tilde G $ uniformly on compact subsets of $ \D $,
	\[
	\vert G_n^{\prime}(0) \vert\big/2 \geq \vert \tilde G^{\prime}(0) \vert\big/2 - \vert G_n^{\prime}(0) - \tilde G^{\prime}(0) \vert\big/2 \geq m - \tau
	\]
	for $ n $ large. Here $ 0 < \tau < m $. On the other hand,
	\[
	\vert G_n^{\prime}(z) - G_n^{\prime}(0) \vert 
	\leq 
	\vert  \tilde G^{\prime}(z) - \tilde  G^{\prime}(0) \vert 
	+ \vert G_n^{\prime}(z) - \tilde G^{\prime}(z) \vert 
	+ \vert G_n^{\prime}(0) - \tilde  G^{\prime}(0) \vert \leq m \big/2 + \epsilon + \epsilon
	\]
	for $ \vert z \vert < \delta $
	and $ n $ large enough. Therefore, if $ 2 \epsilon + \tau < m\big/2 $, then 
	\[
	\left\vert G_n^{\prime}(z) - G_n^{\prime}(0) \right \vert \leq m \big/2 + 2 \epsilon < m - \tau < \vert G_n^{\prime}(0) \vert \big/2
	\]
	for $ \vert z \vert < \delta $ and $ n $ large enough. It follows from the claim that there is a ball of uniform radius, say $ \eta > 0 $ which is contained in the images $ G_n(B(0, \delta)) $. Choose $ p \in B(0, \eta) $. This will serve as a base point in the following way. Let $ \pi_n : \D \longrightarrow D_n = D \setminus \{
    a_n	\} $ be holomorphic coverings such $ \pi_n(0) = p $. Then there exist holomorphic coverings $ \tilde \pi_n : \D \longrightarrow \D^* $ such that    
	\[
	\begin{tikzcd}
	\mathbb{D} \arrow{r}{\tilde \pi_n} \arrow[swap]{dr}{\pi_n} & \mathbb{D}   ^{*} \arrow{d}{G_n} \\
	& D_n
	\end{tikzcd}
	\]
	commutes, i.e., $ G_n \circ \tilde \pi_n = \pi_n $. By locally inverting $ G_n $ near the origin,
	\[
	\tilde \pi_n(0) = G_n^{-1} \circ \pi_n(0) = G_n^{-1}(p).
	\]
	The family $ \{\tilde \pi_n\} $ is normal and admits a convergent subsequence. Let $ \tilde \pi_0 $ be a limit of $ \tilde \pi_{k_n} $. The image $ \tilde \pi_0(\D)  $ can not intersect $ \partial \D $ as otherwise $ \tilde \pi_0(z) = e^{i \theta_0} $ for some $ \theta_0 $. This contradicts the fact that $ \tilde \pi_{k_n}(0) = G_{k_n}^{-1}(p)$ is compactly contained in $ \D $. If $ \tilde \pi_0(\D) $ were to contain the origin, then $ \tilde \pi_0(z) \equiv 0 $ as otherwise $ \tilde \pi_n(\D) $ would also contain the origin by Hurwitz's theorem. The conclusion of all this is that $ \tilde \pi_0 : \D \longrightarrow \D^* $ is non-constant and hence a covering by Lemma~\ref{L:UniformConvergenceOfPunCoveringMpas}. By Lemma~\ref{L:HyperbolicCoveringConvergenc}, $ \pi_n \to \pi_0 $ where $ \pi_0 : \D \longrightarrow \D_a = D \setminus \{a  \}$ is a covering and hence by passing to the limits in 
	\[
	G_n \circ \tilde \pi_n = \pi_n
	\] 
	we get
	\[
	\tilde G \circ \tilde \pi_n = \pi_0.
	\] 
	This shows that $ \tilde G : \D^* \longrightarrow D_a$ is a covering. As noted earlier $ \tilde G^{\prime}(0) > 0 $ and this means that $ \tilde G = G $, the normalized Hurwitz covering at a.    
\end{proof}
\medskip

\begin{lem}\label{L:HurCoverScalingConvergence}
	Let $ D_j $ be the sequence of domains that converges to $ \h $ as in Proposition~\ref{Prop:ConvergenceAhlforsSzegoGarabeidian}. Let $ z_j $ be a sequence converging to $ a \in \h $ satisfying $ z_j \in D_j $ for all $ j $ and let $ G_j $ be the normalized Hurwitz covering of $ D_j$ at the point $ z_j $ and $ G $ be the normalized Hurwitz covering of $ \h$ at $ a $. Then $ G_j $ converges to $ G $ uniformly on compact subsets of $ \D $. 
\end{lem}
\begin{proof}
	Since $ D_j $ converges to $ \h $ in the Hausdorff sense  and $ z_j \to a $, there exist $ r > 0 $ and a point $ p \in \h $ such that $p \in  B(z_j, r) \subset D_j $ for all large $ j $ -- for simplicity we may assume for all $ j $ -- and $p \in  B(a, r) \subset \h $ satisfying $ p \neq z_j $ and $ p \neq a $ and $ G_j $ is locally invertible in a common neighborhood containing $ p $ for all $ j $. Let $ \pi_j  : \D \longrightarrow D_j \setminus \{z_j\}$ be the holomorphic coverings such that $ \pi_j(0) = p$ and $ \pi^{\prime}(0) > 0 $. Let $ \pi_{0 j} : \D \longrightarrow \D^* $ be holomorphic coverings such that $ \pi_j = G_j \circ \pi_{0 j} $. Since $ D_j \setminus \{z_j\} $ converges to $ \h\setminus \{a\} $ in the Carath\'eodory kernel sense, by Hejhal's result \cite{Hejhal}, $ \pi_j $ converges to $ \pi $ uniformly on compact subsets of $ \D $, where $ \pi : \D \longrightarrow \h\setminus \{a\}  $ is the holomorphic covering such that $ \pi(0) = p $ and $ \pi^{\prime}(0) > 0 $. 
	\medskip
	
	Since $ D_j \to \h $ in the Hausdorff sense, any compact set $ K $ with empty intersection with $ \overline \h $ will have no intersection with $ D_j $ for $ j $ large and as a result the family $ \{G_j\} $ is normal. Also note that $ \pi_{0 j}(\D) = \D^* $ for all $ j $, and this implies that $ \{\pi_{0 j}\}$ is a normal family. Let $ G_0 $ and $ \pi_0 $ limits of  $ \{G_j\} $ and $ \{\pi_j\} $ respectively. Now, together with the fact that $ \pi_0 $ is non-constant -- guaranteed by its construction   -- and the identity $ \pi_j  = G_j \circ \pi_{0 j}$, we have $ \pi = G_0 \circ \pi_0 $. This implies that $ G_0 : \D^* \longrightarrow \h_a = \h \setminus \{a \} $ is a covering since $ \pi_0 $ is a covering which follows by Lemma~\ref{L:UniformConvergenceOfPunCoveringMpas}. Since $ G_j(0) = z_j $ and $ G_j^{\prime}(0) > m > 0 $ -- for some constant $ m > 0 $ which can be obtained using the bilipschitz condition of the Hurwitz metric, for instance, see inequality~ (\ref{Eq:bilipchitzCondition}) -- it follows $ G_0(0) = a $  and $ G_0^{\prime}(0) \geq m $. This shows that $ G_0 $ is the normalized Hurwitz covering, in other words $ G_0 = G $. Thus, we showed that any limit of $ G_j $ is equal to $ G $ and this proves that $ G_j $ converges to $ G $ uniformly on compact subsets of $ \D $.
\end{proof}
\begin{lem}\label{L:ConvergenceOfTheHurwitzMetricScaling}
	Let $\{D_j\}$ be the sequence of domains that converges to the half-space $\h$ as in Proposition~\ref{Prop:ConvergenceAhlforsSzegoGarabeidian}. Then $\eta_{D_j}$ converges to $\eta_{\h}$ uniformly on compact subsets of $\h$.  
\end{lem}
\begin{proof}
	Let $K \subset \mathcal{H}$ be a compact subset. Without loss of generality, we may assume that $K \subset D_j$ for all $j$. 
	\medskip
	
	If possible assume that $\eta_{D_j}$ does not converges to $\eta_{\mathcal{H} }$ uniformly on $K$. 
	Then there exist $\epsilon_{0} >0 $, a sequence of integers $\{ k_j \}$ and sequence of points $\{ z_{k_j}\}\subset K \subset D_{k_j}$ such that
	\[
	\vert \eta_{D_{k_j}}(z_{k_j})- \eta_{\mathcal{H}}(z_{k_j}) \vert >\epsilon_0.
	\]      
	Since $K$ is compact, for simplicity, we assume $z_{k_j}$ converges to a point $z_0 \in K$.
	Using the continuity of the Hurwitz metric, we have 
	\[
	\vert \eta_{\mathcal{H}}(z_{k_j})- \eta_{\mathcal{H}}(z_0) \vert < \epsilon_0\big/2
	\] 
	for $j$ large. 
	By the triangle inequality 
	\[
	\vert \eta_{D_{k_j}}(z_{k_j})- \eta_{\mathcal{H}}(z_0) \vert > \epsilon_0\big/2.
	\]
	Since the domains $D_{k_j}$ are bounded there exist normalized Hurwitz coverings $ G_{k_j} $ of $ D_{k_j} $ at $ z_{k_j} $. Using the convergence of $ D_{k_j}\setminus \{z_{k_j}\} $ to $ \h \setminus \{z_0\}  $ in the Carath\'eodory kernel sense we have by Lemma~\ref{L:HurCoverScalingConvergence}, $ G_{k_j} $ converges to the normalized Hurwitz covering map $ G $ of $ \h $ at the point $ z_0 $ uniformly on every compact subset of $ \h $ as $ j \to \infty $. But we have 
	\[
	\eta_{\h}(z_0) = 1\big/G^{\prime}(0).
	\]
	From above we obtain that $ \eta_{D_{k_j}}(z_{k_j}) $ converges to $ \eta_{\h}(z_0) $ as $ j \to \infty $. But from our assumption, we have 
	\[
	\vert \eta_{D_{k_j}}(z_{k_j})- \eta_{\mathcal{H}}(z_0) \vert  > \epsilon_0\big/2
	\]
	and this is a contradiction. Therefore, $ \eta_{D_{k_j}}$ converges to $\eta_{\mathcal{H}}$ uniformly on compact subsets of $\mathcal{H}$. 
\end{proof}
\medskip

\begin{cor}\label{C:AsymptoticHurwitzMetric}
	Let $ D \subset \C$ be a bounded domain. Suppose $ p \in \partial D $ is a $ C^2 $-smooth boundary point. Then 
	\[
	\lim_{z \to p}\eta_D(z)(-\psi(z)) = \eta_{\h}(0)= \vert \omega \vert \big/2.
	\]
\end{cor}
\begin{proof}
	By Proposition~\ref{Pr:ConvergenceOfTheSugawaMetricUnderScaling}, if $ D_j $ is the sequence of scaled domains of $ D $ under the affine transformation $ T_j(z) = (z - p_j)\big/(-\psi(p_j)) $, for all $ z \in D $, where $ \psi $ is a local defining function of $ \partial D $ at $ z = p $ and $ p_j $ is a sequence of points in $ D $ converging to $ p $, then the corresponding sequence of Hurwitz metrics $ \eta_{D_j}  $ of $ D_j $ converges to $ \eta_{\h} $ uniformly on every compact subset of $ \h $ as $ j \to \infty $. In particular, $\eta_{D_j}(0) $ converges to $ \eta_{\h}(0) $ as $ j \to \infty $. Again, we know that $ \eta_{D_j}(z) = \eta_{D}(T_j^{-1}(z)) \vert (T_j^{-1})^{\prime}(z)\vert $. From this it follows that  $ \eta_{D_j}(0) = \eta_{D}(p_j) (-\psi(p_j)) $, and consequently, $ \eta_{D}(p_j) (-\psi(p_j)) $ converges to $ \eta_{\h}(0) $ as $ j \to \infty $. Since $ p_j $ is an arbitrary sequence converging to $ p $, hence we have the following
	\[
	\lim_{z \to p}\eta_D(z)(-\psi(z)) = \eta_{\h}(0) = \vert \omega \vert \big/2.
	\]
	
\end{proof}

\begin{proof}[Proof of Theorem~\ref{T:TheHurwitzMetric}]
From Theorem~\ref{T:contHurM}, it follows that the Hurwitz metric is continuous and as a consequence of Corollary~\ref{C:AsymptoticHurwitzMetric}, we have
\[
\eta_D(z) \approx 1\big/{\rm dist}(z, \partial D).
\]
\end{proof}

\section{A curvature calculation}
\noindent
For a smooth conformal metric $ \rho(z)\vert dz \vert $ on a planar domain $ D $, the curvature 
\[
K_{\rho} = -\rho^{-2} \Delta \log \rho
\]
is a well defined conformal invariant. It $ \rho $ is only continuous, Heins \cite{Heins} introduced the notion of generalized curvatures as follows:
\medskip

For $ a \in D $ and $ r > 0 $, let
\[
T(\rho, a, r) = \left. -\frac{4}{r^2}\left\{\frac{1}{2 \pi}\int_0^{2\pi} \log \rho(a + r e^{i \theta})d\theta - \log \rho(a)\right\}\right/ \rho^2(a).
\]
The $ \liminf_{r \to 0} T(\rho, a , r)$ and $ \limsup_{r \to 0} T(\rho, a , r)$ are called generalized upper and lower curvature of $ \rho(z)\vert dz \vert $. Our aim is to give some estimates for the quantities $ T(\rho, a , r) $ for $ r > 0 $ and $ \rho = q_D $ or $ \eta_D $. 
\begin{lem}
	Let $ T_j : D \longrightarrow D_j $ be given by 
	\[
	T_j(z) = \frac{z - p_j}{- \psi(p_j)}
	\]
	where $ D $ and $ p_j \to p $ as before. Then 
	\[
	T\left(\rho, p_j , r\vert \psi(p_j) \vert \right) = T\left((T_j)_*\rho, 0 , r\right)
	\]
	for $ r > 0 $ small enough. Here $ (T_j)_*\rho $ is the push-forward of the metric $ \rho $.
\end{lem} 
\begin{proof}
	Computing,
	\begin{align*}
	&T\left((T_j)_*\rho, 0 , r\right)\\
	 &=  \left. -\frac{4}{r^2} \left\{\frac{1}{2 \pi}\int_0^{2\pi} \log((T_j)_*\rho)( re^{i\theta}d) d\theta - \log((T_j)_*\rho)(0)\right\}\right/((T_j)_*\rho)^2(0)\\
	&= \left.-\frac{4}{r^2} \left\{\frac{1}{2 \pi}\int_0^{2\pi} \log\rho (T_j^{-1}(re^{i\theta}))\vert (T_j^{-1})^{\prime}(re^{i\theta}) \vert d\theta - \log\rho ((T_j^{-1})(0))\vert (T_j^{-1})^{\prime}(0) \vert\right\}\right/\rho ((T_j^{-1})(0))^2\vert (T_j^{-1})^{\prime}(0) \vert^2.
	\end{align*}
	Since $ (T_j^{-1})^{\prime}(z) = -\psi(p_j) $ for all $ z \in D $, we have
	\[
	T\left((T_j)_*\rho, 0 , r\right) = \left. -\frac{4}{r^2\vert \psi(p_j) \vert^2} \left\{\frac{1}{2 \pi}\int_0^{2\pi} \log\rho (p_j + r\vert \psi(p_j) \vert e^{i\theta}) d\theta - \log\rho(p_j )\right\}\right/\rho^2(p_j ).
	\]
	This implies
	\[
	T\left(\rho, p_j , r\vert \psi(p_j) \vert \right) = T\left((T_j)_*\rho, 0 , r\right).
	\]
\end{proof}
\begin{lem}
	Fix $ r_0 $ small enough and let $ \epsilon > 0 $ arbitrary. Then
	\[
	-4 - \epsilon < T\left(\rho, p_j , r_0\vert \psi(p_j) \vert \right) < -4 + \epsilon.
	\]
	for $ j $ large. Consequently, for a fixed $ r $ with $ 0 < r < r_0 $,
	\[
	\lim_{j \to \infty}T(\rho, p_j, \vert \psi(p_j) \vert r) = -4
	\]
	
\end{lem}
\begin{proof}
	From the lemma above, it is enough to show the inequality for $ T\left((T_j)_*\rho, 0 , r_0\right) $. Recall that the scaled Sugawa metric and Hurwitz metric, i.e., $ q_{D_j} $ and $ \eta_{D_j} $ converge to $ q_{\h} $ and $ \eta_{\h} $ respectively. The convergence is uniform on compact subsets of $ \h $ and note that $ q_{\h} = \eta_{\h} $. By writing $ \rho_j $ for $ q_{D_j} $ or $ \eta_{D_j} $, and $ \rho_{\h} $ for $ q_{\h} = \eta_{\h} $, we get
	\[
	T(\rho_{D_j}, 0, r) \to T(\rho_{\h}, 0, r)  
	\]
	for a fixed $ r $ as $ j \to \infty $. That is
	\[
	T(\rho_{\h}, 0, r) - \epsilon\big/2 < T(\rho_{D_j}, 0, r) < T(\rho_{\h}, 0, r) + \epsilon\big/2
	\]
	for $ j $ large. Also recall that since the curvature of $ \rho_{\h} $ is equal to $ -4 $, there exists $ r_1 > 0 $ such that 
	\[
	-4 - \epsilon\big/2 < T(\rho_{\h}, 0, r) < -4 + \epsilon\big/2
	\]
	whenever $ r < r_1 $. Now, choose $ r_0 = r_1/2 $, then 
	there exists $ j_0 $ depending on $ r_0 $ such that
	\[
	-4 - \epsilon < T\left(\rho_{D_j}, 0 , r_0\right) < -4 + \epsilon.
	\]
	for all $ j \geq j_0 $.
	This completes the proof.
\end{proof}


\end{document}